\documentclass[a4paper,11pt]{article}

\usepackage{amsfonts,amsmath,color,amssymb,graphicx}

\usepackage{pstricks}
\usepackage{pst-coil}
\usepackage{latexsym}
\usepackage{amssymb}
\usepackage{amsmath}
\usepackage{amsfonts}
\usepackage{amsthm}
\usepackage{graphicx}
\usepackage{enumerate}
\usepackage{geometry}
\usepackage{bm}
\usepackage{comment}

\usepackage{xspace}

\newcommand{\note}[1]%
{\noindent\centerline{\fbox{\parbox{.9\textwidth}{\textbf{#1}}}}}
\newcommand{\snote}[1]%
{\fbox{\textbf{#1}}}

\newtheorem{proposition}{Proposition}[section]

\newtheorem{example}{Example}[section]

\newcommand{\RR}{{\mathbb{R}}}

\begin{document}
\bibliographystyle{plain}

\title{On the decay of the inverse of matrices \\ that are sum of Kronecker 
products}
\author{Claudio Canuto$^a$, Valeria Simoncini$^b$ 
and Marco Verani$^c$
}
\maketitle
\begin{center}
{\small
$^a$Dipartimento di Scienze Matematiche,
Politecnico di Torino,
Corso Duca degli Abruzzi 24,
I-10129 Torino, Italy\\ {\tt ccanuto@polito.it} 
\vskip 0.1cm
$^b$Dipartimento di Matematica,
Universit\`a di Bologna,
Piazza di Porta San Donato  5, I-40127 Bologna, Italy\\
{\tt valeria.simoncini@unibo.it}
\vskip 0.1cm
$^c$MOX - Dipartimento di
Matematica, Politecnico di Milano, via Bonardi, 9, I-20133 Milano, Italy \\
{\tt marco.verani@polimi.it}
}
\end{center}

\begin{abstract}
Decay patterns of
matrix inverses have recently attracted considerable interest,
due to their relevance in numerical analysis,
and in applications requiring matrix
function approximations.
 In this paper we analyze the decay pattern of the inverse
of banded matrices in the form $S=M \otimes I_n + I_n \otimes M$ where $M$ is tridiagonal, symmetric
and positive definite, $I_n$ is the identity matrix, and $\otimes$ stands for the Kronecker product. 
It is well known that the inverses of banded matrices exhibit an exponential
decay pattern away from the main diagonal.  However, the entries in $S^{-1}$ 
show a non-monotonic decay, which is not caught
by classical bounds. By using an alternative expression for $S^{-1}$, we
derive computable upper bounds that
 closely capture the actual behavior of its entries. We also show that similar estimates
can be obtained when $M$ has a larger bandwidth, or when the sum of Kronecker
products involves two different matrices.
Numerical experiments illustrating the new bounds are also reported.
\end{abstract}

\section{Introduction}
We consider nonsingular matrices 
$S$ of size $n^2\times n^2$ that can be written as
\begin{eqnarray}\label{eqn:main}
S = M \otimes I_n + I_n \otimes M ,
\end{eqnarray}
where $M$ is an $n\times n$ banded symmetric and positive
definite matrix (SPD) and $\otimes$ is the Kronecker product; 
here $I_n$ is the identity matrix of size $n$. 
Matrices in this form may arise for instance in the discretization of 
two-dimensional partial differential equations by means of finite difference,
spectral or finite element methods.
We say that a symmetric matrix $A$ is $b$-banded if its entries
$A_{ij}$ satisfy $A_{ij}=0$ for $|i-j|>b$.
In the following, we shall 
mainly focus on the case when $M$ is tridiagonal, so that $b=1$.
As a consequence of $M$ being banded, $S$ will also be banded, although
its bandwidth will be much larger: if $b$ is the bandwidth of $M$,
then $b\cdot n$ will be the bandwidth of $S$.

We are interested in exploring the magnitude pattern of the entries $(S^{-1})_{ij}$. It is
well-known that although the {\it inverse} of a banded matrix
is full in general - and in particular it is
not banded - its entries exponentially decay 
as their location deviates from the main diagonal; such a
decay pattern 
was analyzed  in detail in
\cite{SDWFMPWS84} for 
$S$ a general symmetric positive definite $b$-banded matrix. Indeed,
it was shown in 
\cite{SDWFMPWS84} that 
\begin{eqnarray}\label{eqn:demko}
|(S^{-1})_{ij} | \le \gamma q^{\frac{|i-j|}{b}}
\end{eqnarray}
where $\kappa$ is the condition number of $S$,
$q=(\sqrt{\kappa}-1)/(\sqrt{\kappa}+1)$,
$\gamma = \max\{\lambda_{\rm min}(S)^{-1}, \hat \gamma\}$, and
$\hat\gamma = (1+\sqrt{\kappa})^2/(2\lambda_{\max}(S))$; in this
bound the diagonal elements of $S$ are assumed not to be
greater than one. Here and in the following, 
$\lambda_{\min}(\cdot), \lambda_{\max}(\cdot)$ denote the smallest and largest
eigenvalues of the given symmetric matrix.

Decay patterns have attracted considerable interest in the scientific
computing community
in the last two decades,
due to their relevance in the context of linear system preconditioning 
\cite{Benzi.Golub.99}, \cite{Benzi.prec.02},
low-rank approximation strategies such as hierarchical matrices, 
wavelets etc. \cite{Wojtabook.97}, \cite{Bebendorf.Hackbusch.03},  
and in a large variety of applications requiring
matrix function approximations, such as electronic structure calculations,
complex networks, robotics, etc.; see, e.g.,  
\cite{Benzi.Boito.14},
\cite{Benzi2007},\cite{BenziSIREV.13},\cite{Benzi.Boito.10},\cite{Maslenetal.98}, %
 and the references therein.

A large amount of
literature has focused on the inverse entries of (irreducible) tridiagonal matrices
for which explicit formulas and recurrence relations
are now available; see, e.g., \cite{Vandebriletalbook1.08}, \cite{MeurantJul.1992}
and their references.
Some of these results can be generalized to {\it block} tridiagonal cases, of which
(\ref{eqn:main}) is a particular case for $M$ tridiagonal, however accurate estimates for the entries
have only been obtained under more restrictive assumptions \cite{Nabben.99}. 
In \cite{MeurantJul.1992}, for instance, the case of the discretization of the two-dimensional
Poisson operator was considered, which corresponds to (\ref{eqn:main})
with $M$ SPD, tridiagonal and with constant coefficients
(see Example \ref{ex:1} below). 


A key point of the matrices in the form (\ref{eqn:main}) is
that
the decay of the entries of its inverse is not monotonic away
from the diagonal. In fact, the entries decay in a way that
recalls a sinusoidal behavior converging to zero. 
We report in Figure \ref{fig:laplinv}
a typical such pattern, obtained for $M=-{\rm tridiag}(1,-\underline{2},1)$
(here and later in the paper, the underlined number lies on the matrix diagonal),
 corresponding
to the finite difference discretization of the two-dimensional negative Laplace
operator $-(u_{xx} + u_{yy})$ in the domain $[0,1]\times [0,1]$.
This non-monotonic behavior has been observed in the literature 
(\cite{MeurantJul.1992}),
and  explained in detail for the case of the discrete Laplacian, for which
precise estimates are available \cite{Bramble.Thomee.69}, \cite{McAllister.Sabotka.73}, 
\cite{Vejchodsky.Solin.07};
bounds stemming from an algebraic analysis were also determined in \cite{MeurantJul.1992}.
The situation is far less understood when $M$ is any tridiagonal SPD
matrix, or more generally any banded SPD matrix. Clearly, classical bounds such
as the one in (\ref{eqn:demko}) cannot catch this non-monotonic pattern,
although its detection can be crucial in sparsity-based approximation procedures.
In this paper we derive bounds that closely capture this
non-monotonic behavior, which is typical of matrices in the form (\ref{eqn:main}).
\begin{figure}[htb]
\centering
\includegraphics[width=2.5in,height=2.5in]{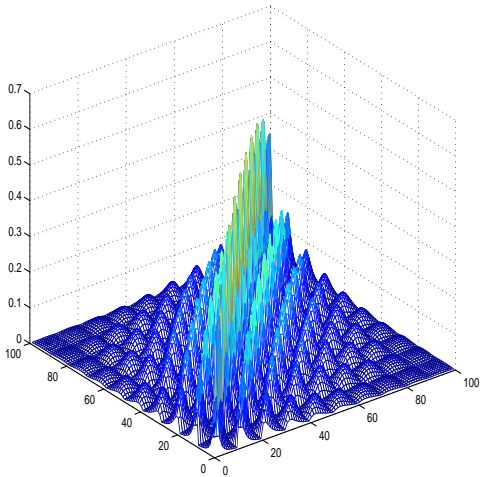}
\caption{Pattern of the
inverse of the 2D Laplace $100\times 100$ matrix in the unit square.\label{fig:laplinv}}
\end{figure}
In particular, we show that the decaying oscillation observed in practice in $|(S^{-1})_{i,j}|$ for
$i,j=1, \ldots, n^2$,
strongly depends on, and can be bounded by,
 the ``mesh'' distance between the two indices $i, j$ when each of them is
represented in a natural $n\times n$ grid.  In section \ref{sec:decay}
we provide sharp estimates,
followed by easily computable more qualitative bounds; the latter can be incorporated, for instance,
 in numerical thresholding strategies during a sparsity-oriented approximation of the matrix inverse 
(see, e.g., section \ref{sec:Cholfactor}).

In section \ref{sec:gen} we shall extend our results to banded SPD matrices, and to
the more general case 
\begin{eqnarray}\label{eqn:M1M2}
S_g := M_1 \otimes I_n + I_n \otimes M_2 ,
\end{eqnarray}
where $M_1$ and $M_2$ are symmetric tridiagonal matrices stemming, for instance,
from the discretization by finite differences of a self-adjoint
separable second order differential operator on a stretched rectangular domain,
or of an operator with different coefficients in the two space directions;
see, e.g., \cite{Levequebook.07}.

\section{Decay of the entries of the inverse of $S$ for $M$ tridiagonal}\label{sec:decay}
Let $X=S^{-1}$, and write $X=[x_1, \ldots, x_t, \ldots, x_{n^2}]$.
A simple but key observation is that each column $t$ of the inverse $X$ is the solution
to the linear system
$$
S x_t = e_t,
$$ 
where $e_t$ is the $t$-th column of $I_{n^2}$. Let us
define ${\cal W}_t$ to be the matrix such that $w_t = {\rm vec}({\cal W}_t)$ with
$w_t\in\RR^{n^2}$ and ${\cal W}_t \in \RR^{n\times n}$ (the ``vec'' operation stacks
the columns of ${\cal W}_t$ one below the other).
With this notation, and
using the Kronecker form of $S$, the system above is equivalent to
$$
M {\cal X}_t + {\cal X}_t M = {\cal E}_t.
$$
Since $e_{t} ={\rm vec}({\cal E}_t)$, $t=1, \ldots, n^2$, the 
matrix ${\cal E}_t$ has a single nonzero element $({\cal E}_t)_{ij}$, with indices
$j=\lfloor (t-1)/n\rfloor +1$, $i=t-n\lfloor (t-1)/n\rfloor $, 
$i,j\in \{1, \ldots, n\}$. Therefore, we can write
${\cal E}_t= {\cal E}_{i+n(j-1)}=e_ie_j^\top$.
 
The derivation above shows
that the $n^2$ entries of each column of $S^{-1}$, properly reordered,
correspond to the $n\times n$ entries of the solution matrix
to a Lyapunov equation. In Figure \ref{fig:lyapsol} we report the
pattern of ${\cal E}_t$ (left) and of ${\cal X}_t$ (right) for $t=26$ when
$S$ is the finite difference discretization of the two-dimensional Laplace
operator in the unit square. Note that because of the isotropy property of the
operator, the forcing term (the right-hand side) diffuses in a similar way
in both directions; see also a related discussion in \cite[section 4.1]{MeurantJul.1992}.

We next exploit the closed form of the Lyapunov solution to derive bounds for the
entries of $S^{-1}_{:,t}={\rm vec}({\cal X}_t)$ for each $t=1, \ldots, n^2$. 
%
Let $j=\lfloor (t-1)/n\rfloor +1$, $i=t-n\lfloor (t-1)/n\rfloor $.
Since $M$ is positive definite, the solution can be written as (see, e.g., \cite{Horn.Johnson.91})
\begin{eqnarray*}
{\cal X}_t & = &
\frac{1}{2\pi}
\int_{-\infty}^{\infty} (\imath \omega I + M)^{-1} {\cal E}_t 
(\imath \omega I + M)^{-*} {\rm d}\omega\\
& =&
\frac{1}{2\pi}
\int_{-\infty}^{\infty} (\imath \omega I + M)^{-1} e_i e_j^\top
(\imath \omega I + M)^{-*} {\rm d}\omega
\equiv
\frac{1}{2\pi}
\int_{-\infty}^{\infty} z_i z_j^* {\rm d}\omega
\end{eqnarray*}
where $z_i= (\imath \omega I + M)^{-1} e_i$.
We are interested in estimating the $k$-th entry of the $t$-th column of the inverse $S^{-1}$.
Using 

\begin{figure}[htb]
\centering
\includegraphics[width=2.5in,height=2.5in]{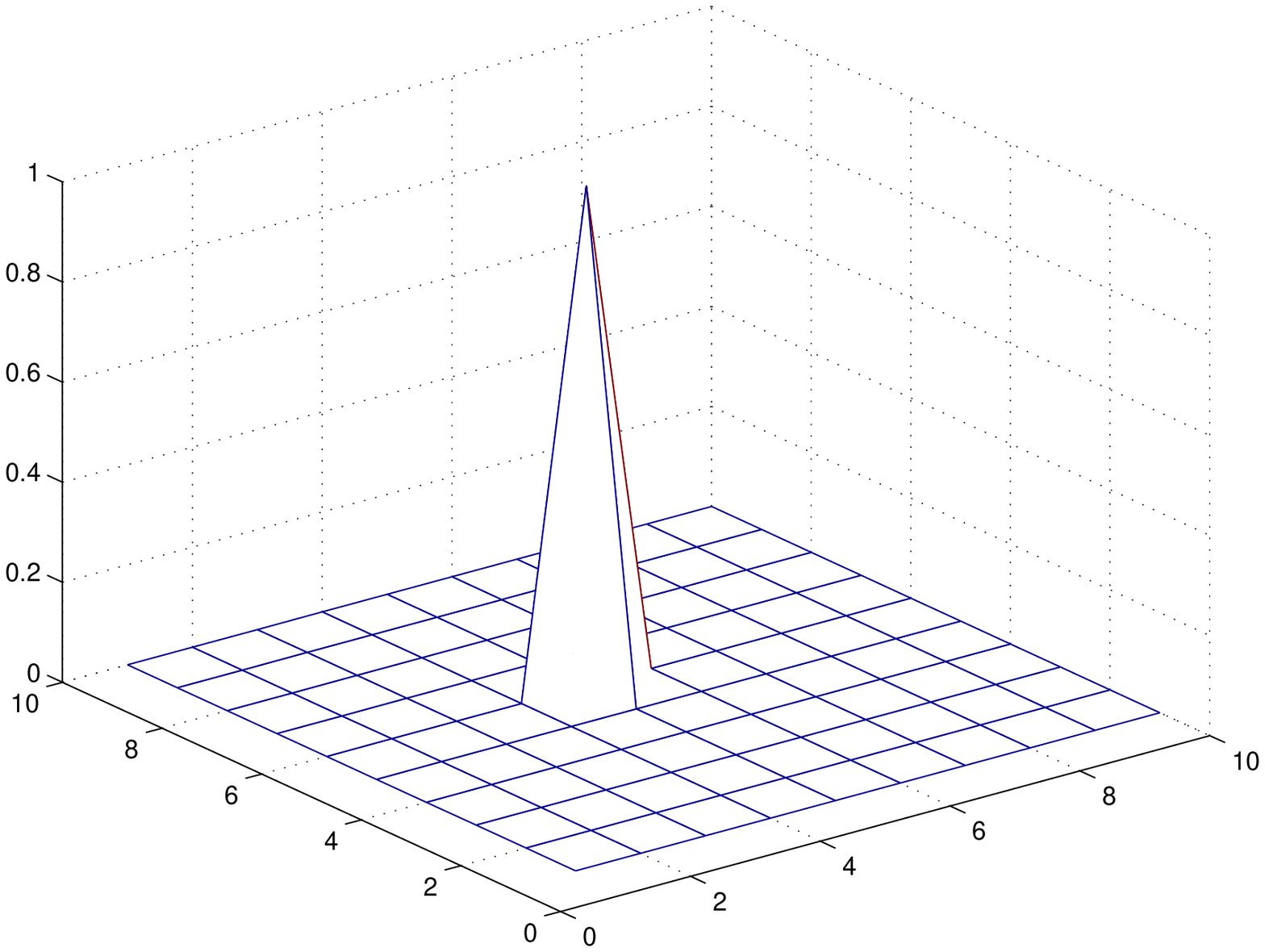}
\includegraphics[width=2.5in,height=2.5in]{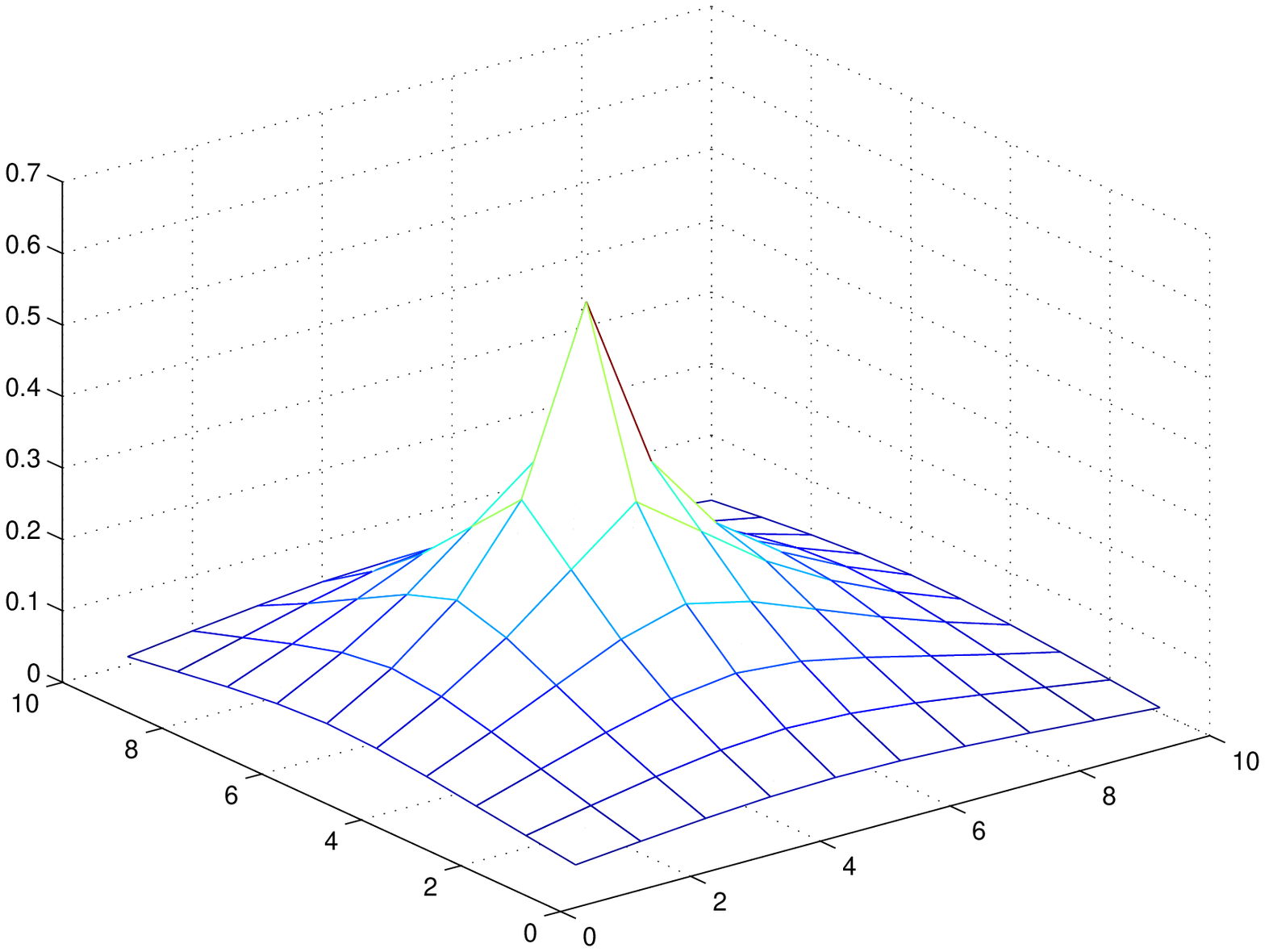}
\caption{Left: Column 46 of the $100\times 100$
identity matrix, represented on a $10\times 10$ grid. Right:
Column 46 of the inverse Laplacian, represented on a $10\times 10$ grid.\label{fig:lyapsol}}
\end{figure}

\begin{eqnarray}\label{eqn:indices}
m=\lfloor (k-1)/n\rfloor+1, \quad \ell=k-n \lfloor (k-1)/n\rfloor,
\end{eqnarray}
 this corresponds
to estimating the entry $({\cal X}_t)_{\ell, m} = e_\ell^\top {\cal X}_t e_m$ of 
${\cal X}_t$,
that is
$$
(S^{-1})_{k,t} = 
(S^{-1})_{\ell+n(m-1),t} = 
e_\ell^\top {\cal X}_t e_m,  \quad \ell, m\in \{1, \ldots, n\} .
$$
By varying $m, \ell \in \{1, \ldots, n\}$ all the elements of the $t$-th
column, $(S^{-1})_{:,t}$ are obtained.
We have
$$
e_\ell^\top {\cal X}_t e_m  = 
\frac{1}{2\pi}
\int_{-\infty}^{\infty} e_\ell^\top z_i(\omega) z_j(\omega)^*e_m {\rm d}\omega ,
$$
so that
\begin{eqnarray}\label{eqn:Xlm}
|e_\ell^\top {\cal X}_t e_m|  \le 
\frac{1}{2\pi}
\int_{-\infty}^{\infty} |e_\ell^\top z_i(\omega)|\, |z_j(\omega)^*e_m|\, {\rm d}\omega.
\end{eqnarray}
Since $e_\ell^\top z_i(\omega) = e_\ell^\top (\imath \omega I + M)^{-1} e_i$,
the first term in the integrand above is the absolute value of
the $(\ell, i)$ entry  of the inverse of tridiagonal matrix
$(\imath \omega I + M)$. In the following we shall bound each of the two
integrand terms, and then we will estimate the obtained integral.

Let $\lambda_{\min}, \lambda_{\max}$ be the extreme eigenvalues of $M$,
and let $\lambda_1 = \lambda_{\min} + \imath \omega$, 
$\lambda_2 = \lambda_{\max} + \imath \omega$. 
The matrix $\imath \omega I + M$ is a purely imaginary shifted version of the tridiagonal
matrix $M$, and its inverse shows a decreasing pattern, in spite of the
complex shift. 
While estimates for $|e_\ell^\top (\imath \omega I + M)^{-1} e_i|$ are well known
for $\omega=0$ (see, e.g., \cite{SDWFMPWS84,Nabben.99,MeurantJul.1992}), upper bounds for $\omega \ne 0$
are less so. 
Upper bounds for $|e_\ell^\top (\imath \omega I + M)^{-1} e_i|$, $\omega\ne 0$ were
given by Freund in \cite[Theorem 6]{Freund1989a}, and we recall
this result for future reference.

\begin{proposition}\label{prop:Freund}
Assume $M$ is symmetric positive definite and $b$-banded. Let 
$a=(\lambda_1+\lambda_2)/(\lambda_2-\lambda_1)$, and $R>1$ be defined as
$R=\alpha + \sqrt{\alpha^2-1}$, 
with $\alpha=(|\lambda_1|+|\lambda_2|)/|\lambda_2-\lambda_1|$.
Then
$$
|e_\ell^\top (\imath \omega I + M)^{-1}e_i| \le \frac{2R}{|\lambda_1-\lambda_2|}
B(a) \left (\frac{1}{R}\right )^{\frac{|\ell -i|}{b}}, \quad \ell\ne i,
$$
where, writing 
$a=\alpha_R\cos(\psi)+\imath \beta_R\sin(\psi)$,
$$
B(a) := \frac {R}{   \beta_R\sqrt{\alpha_R^2-\cos^2(\psi)}(\alpha_R +
\sqrt{\alpha_R^2-\cos^2(\psi)})},
$$
with $\alpha_R=\frac 1 2 (R+\frac 1 R)$ and
$\beta_R=\frac 1 2 (R-\frac 1 R)$.
\end{proposition}

Clearly, $R=R(\omega)$. We omit this explicit dependence in the following.
Figure \ref{fig:Freund} reports two typical behaviors of the bound in
Proposition \ref{prop:Freund}, for the pentadiagonal matrix in Example
\ref{ex:penta}. The plots refer to $\omega=0.10$ (left) and $\omega=10$ (right): while
the bound accurately captures the slope for large $\omega$, this is in
general less so for small $\omega$.
This difference in accuracy in general
may affect the accuracy of our estimates, especially when 
a bandwidth $b$ greater than one is used (here $b=2$).

\begin{figure}[htb]
\centering
\includegraphics[width=2.5in,height=2.5in]{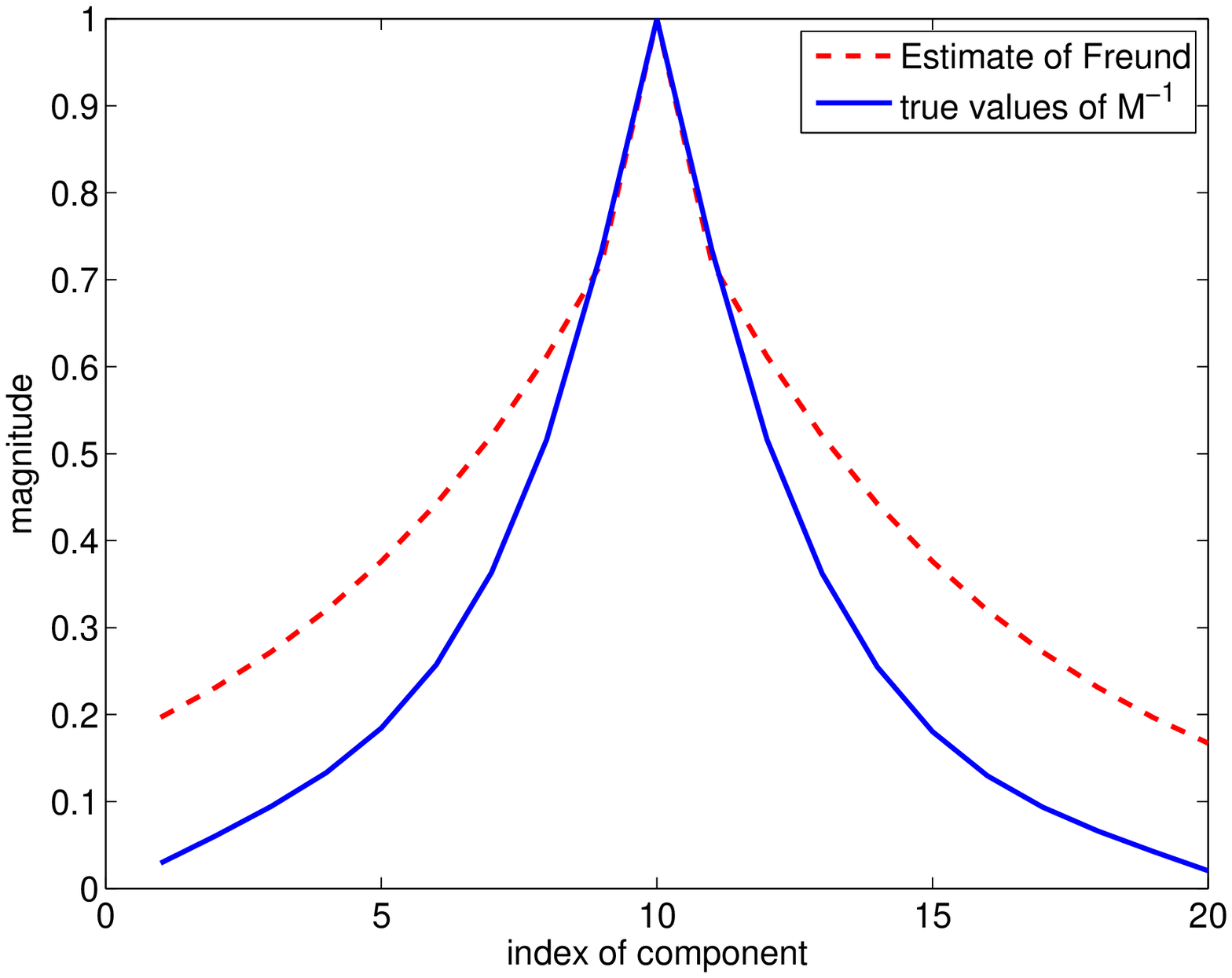}
\includegraphics[width=2.5in,height=2.5in]{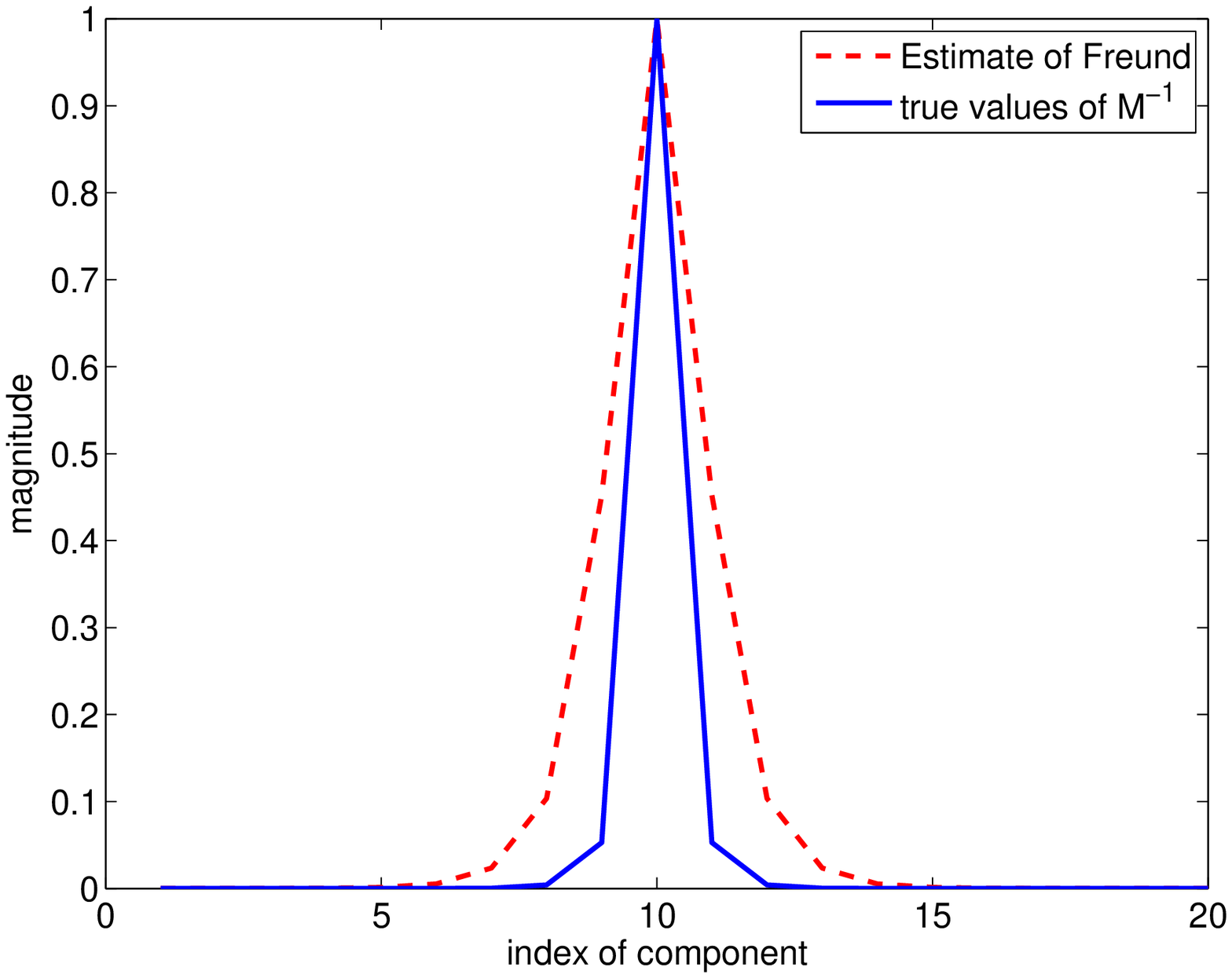}
\caption{Typical estimate of Proposition \ref{prop:Freund} (column 10)
for the inverse of the pentadiagonal matrix of Example \ref{ex:penta}.
Left: $\omega=0.1$. Right: $\omega=10$.
\label{fig:Freund}}
\end{figure}

Since in this section we assume that $M$ is tridiagonal, we shall use the result above for $b=1$;
the case $b>1$ is explored in section \ref{sec:gen}.
We prove our bound in two steps. In the first step (Proposition \ref{prop:step1}), 
we estimate the entries
in terms of an integral, which can be easily estimated numerically; the
results appear to be quite accurate in our examples.
In the second step (Proposition \ref{prop:step2}), we complete the upper bound by estimating the
integrals, thus incurring in additional inaccuracies. 
The final bound (Propositions \ref{prop:step2}-\ref{prop:final})
should be considered as a qualitative estimate for the entries pattern.

\begin{proposition}\label{prop:step1}
 For $k, t\in \{1, \ldots, n^2\}$,
let 
$$
j=\lfloor (t-1)/n\rfloor+1,\quad i=t-n \lfloor (t-1)/n\rfloor ,
$$
 and $\ell, m$ as in (\ref{eqn:indices}).
With the notation above, the following holds.

i) If $i\ne \ell$ and $j\ne m$, then
$$
|(S^{-1})_{k,t}|  
\le
\frac{1}{2\pi}
\frac{64}{|\lambda_{\max}-\lambda_{\min}|^2} \int_{-\infty}^{\infty} 
 \left ( \frac{R^2}{(R^2-1)^2}\right )^2
\left ( \frac{1}{R}\right )^{|i-\ell|+|j-m|-2} {\rm d}\omega ;
$$ 
ii) If either $i= \ell$  or $j= m$, then
$$
|(S^{-1})_{k,t}|  
\le
\frac{1}{2\pi}
\frac{8}{|\lambda_{\max}-\lambda_{\min}|} \int_{-\infty}^{\infty} 
\frac 1 {\sqrt{\lambda_{\min}^2 + \omega^2}}
 \frac{R^2}{(R^2-1)^2}
\left ( \frac{1}{R}\right )^{|i-\ell|+|j-m|-1} {\rm d}\omega ;
$$ 
iii) If both $i= \ell$  and $j= m$, then
$$
|(S^{-1})_{k,t}|  
\le
\frac{1}{2\pi}
\int_{-\infty}^{\infty} \frac 1 {\lambda_{\min}^2 + \omega^2} {\rm d}\omega = \frac 1 {2\lambda_{\min}}.
$$
\end{proposition}

\begin{proof}
To prove i),
we recall that $|(S^{-1})_{k,t}|  = |e_\ell^\top {\cal X}_t e_m|$ so that
(\ref{eqn:Xlm}) holds, and we notice 
that
$\alpha_R^2-1 = \beta_R^2$ and $\alpha_R+\beta_R=R$. Moreover,
$\frac{1}{\sqrt{\alpha_R^2-\cos^2(\psi)}} \le
1/(\sqrt{\alpha_R^2-1})$. Therefore,
\begin{eqnarray*}
B(a) \le \frac{R}{\beta_R\sqrt{\alpha_R^2-1}(\alpha_R+\sqrt{\alpha_R^2 -1})} =
\frac{R}{\beta_R^2(\alpha_R+\beta_R)} = \frac{1}{\beta_R^2}, 
\end{eqnarray*}
so that, using Proposition \ref{prop:Freund},
$$
|e_\ell^\top (\imath \omega I + M)^{-1}e_i| \le \frac{2R}{|\lambda_1-\lambda_2|}
\frac{4R^2}{(R^2-1)^2} \left (\frac 1 R\right )^{|\ell -i|}.
$$
Substituting the estimate for each of the two integrand terms in (\ref{eqn:Xlm}), we obtain
$$
|e_\ell^\top {\cal X}_t e_m| \le
\frac{1}{2\pi}
\frac{64}{|\lambda_1-\lambda_2|^2} \int_{-\infty}^{\infty} 
R^2 
\frac{R^2}{(R^2-1)^2} 
\frac{R^2}{(R^2-1)^2} 
\left (\frac 1 R\right )^{|\ell -i|}
\left (\frac 1 R\right )^{|m -j|} {\rm d}\omega,
$$
from which the result follows.

As of ii) we only need to notice that if, for instance, $\ell=i$, then
\begin{eqnarray}\label{eqn:equal}
|e_\ell^\top (\imath \omega I + M)^{-1} e_i| \le \frac 1 {|\lambda_{\min}+\imath \omega|}.
\end{eqnarray}
Substituting into the integral, the bound follows as in the previous case.

For the case iii), the bound (\ref{eqn:equal}) can be used for both pairs of
indices, and the final bound follows.
\end{proof}

We next report on a few examples showing the quality of the estimates in
Proposition~\ref{prop:step1}. As one might expect, the factor in front of
the integral slightly deteriorates the estimate, while qualitatively
the decay of the entries in the inverse matrix is perfectly captured. 
We observe that the bounds ii) and iii) can be very loose because of
the estimate's weakness in (\ref{eqn:equal}). This can be clearly observed in the
numerical experiments, below, where the inaccuracy of our estimate is
more pronounced in correspondence with the highest peaks, namely for
$\ell=i$ and/or $m=j$.
In all examples, the matrix $M$ was scaled by its diagonal,
so as to have entries all not greater than one.
Technically, the integral appearing in the bound was estimated
using the adaptive Gauss-Kronrod quadrature rule
(Matlab function {\tt quadgk}).

\begin{figure}[hbt]
\centering
\includegraphics[width=2.5in,height=2.5in]{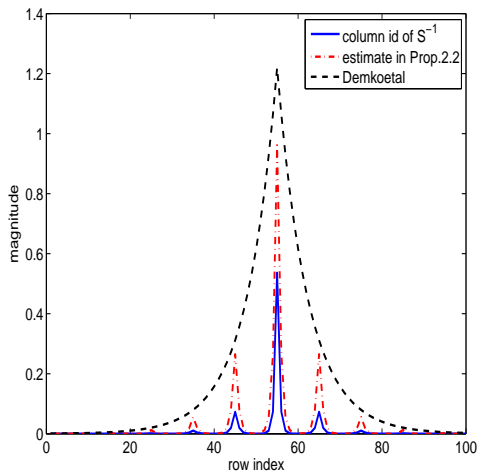}
\caption{Example \ref{ex:0}. Values of column $t=55$ of $S^{-1}$ (solid), estimates for that
column as in Prop.\ref{prop:step1} (dashed), and classical bound in \cite{SDWFMPWS84} (dash-dotted).
\label{fig:ex0}}
\end{figure}

\vskip 0.1in
\begin{example}\label{ex:0}
{\rm
We consider the symmetric diagonally dominant matrix
$$
M = {\rm tridiag}(-0.5,\underline{2},-0.5) \in\RR^{10\times 10}.
$$
As a sample, we consider column $t=55$ (corresponding to the node at
$i=5$ and $j=4$ in the reference grid), and we compute the upper estimate 
for $(S^{-1})_{:,t}$ as the row index varies. 
Figure \ref{fig:ex0} shows the accuracy of the estimates in Proposition~\ref{prop:step1} (dashed
curve), compared with the actual values (solid curve) of column 55.
The estimate is able to capture the highly oscillating decay of the entries
of $S^{-1}$ although, as already mentioned, the peaks are somewhat overestimated.
For completeness, the bound (\ref{eqn:demko}) from \cite{SDWFMPWS84} is also reported; for this
bound, we took into account that $S$ has bandwidth $b=n=10$.
We observe that this classical bound provides a good envelope of the actual decay, although,
as expected, it misses the oscillation pattern. We also note that the classical bound matches the
peaks of our new bound, showing that the classical
predicted decay is obtained for either $i=\ell$ or $j=m$, corresponding to the
rows and columns in the grid most slowing decaying (see also Figure \ref{fig:lyapsol}).
}
\end{example}

\begin{figure}[hbt]
\centering
\includegraphics[width=2.5in,height=2.5in]{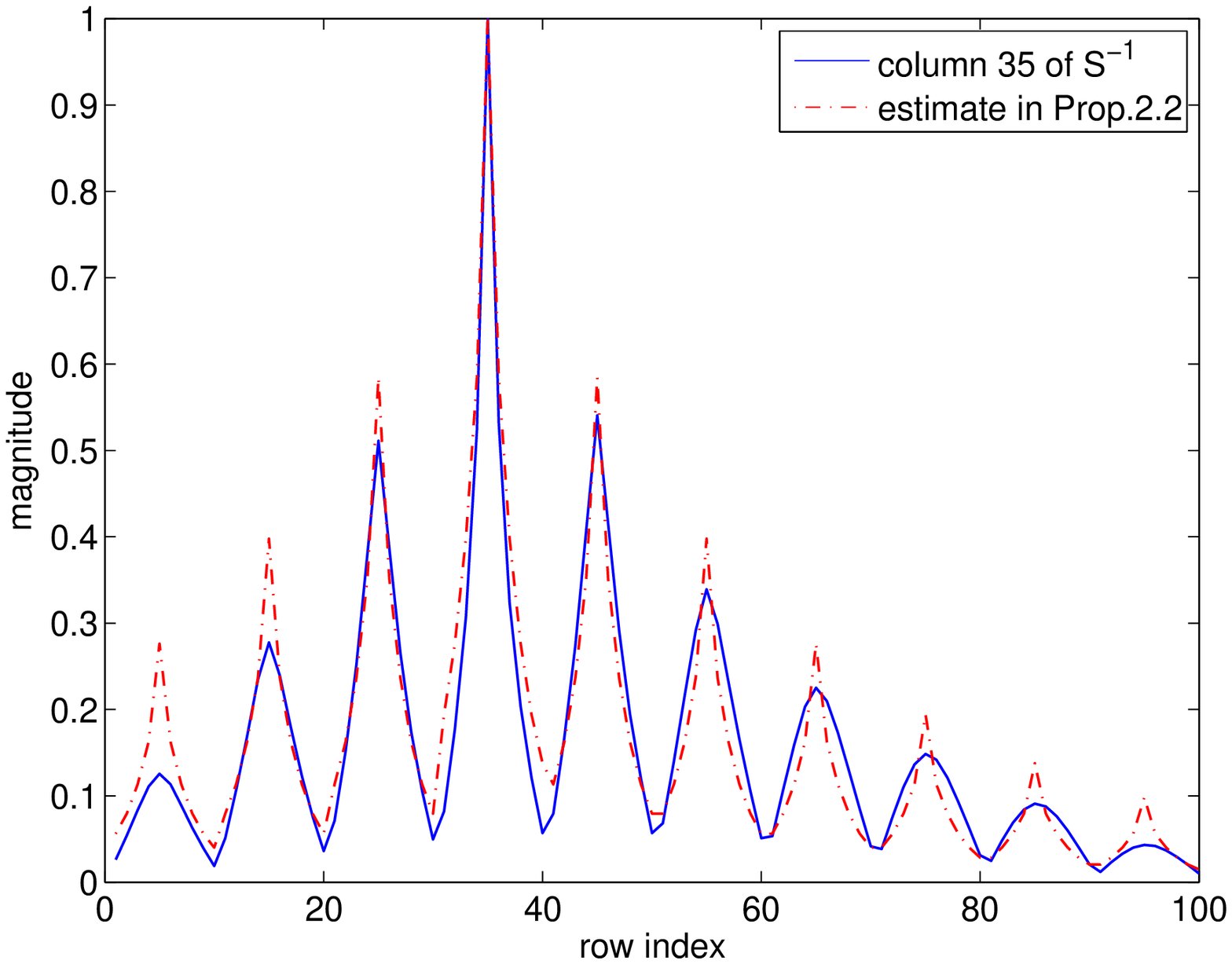}
\includegraphics[width=2.5in,height=2.5in]{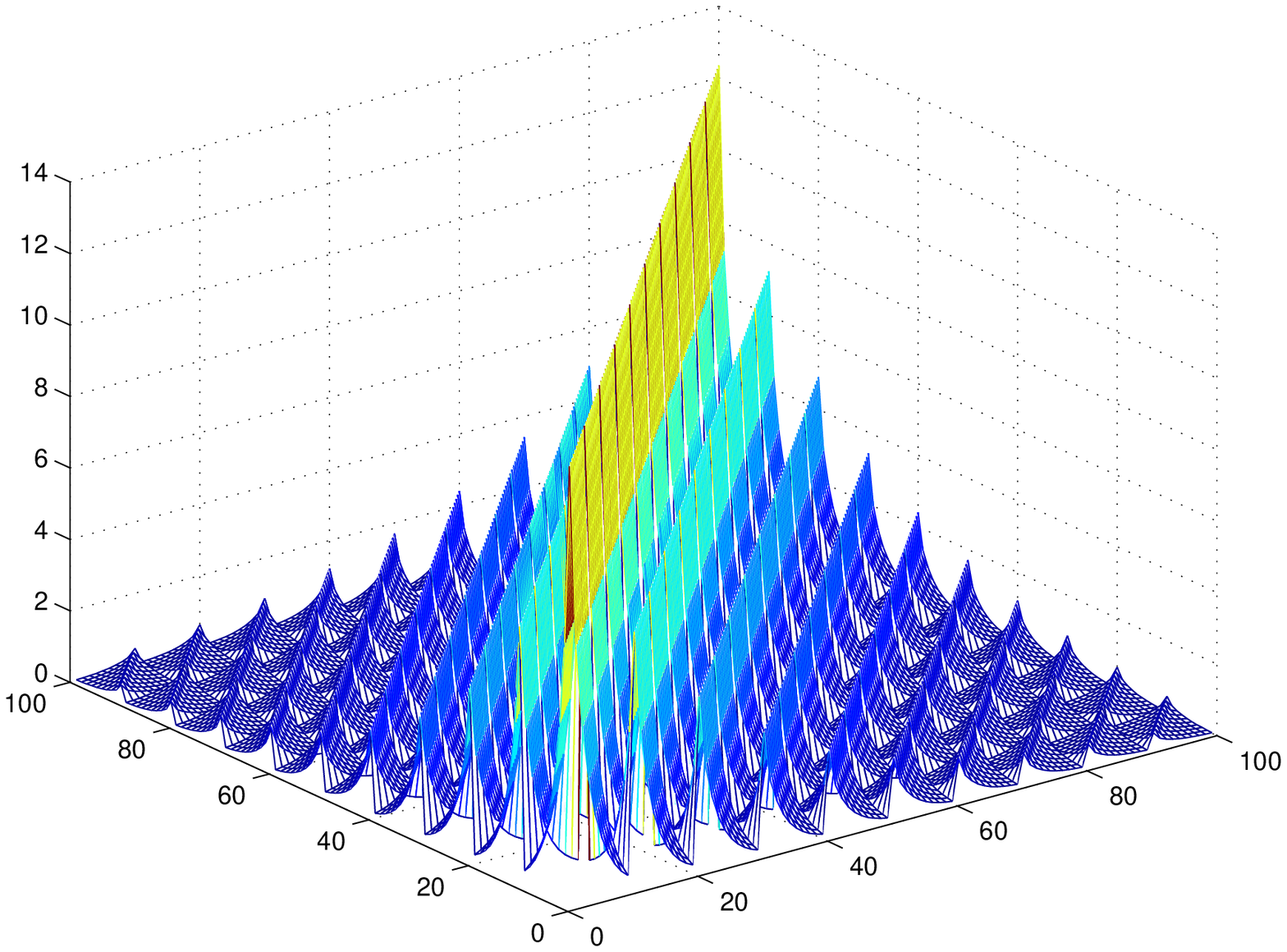}
\caption{Example \ref{ex:1}. 
Left: Components of column $t=35$ of the inverse of the 2D Laplace 
$100\times 100$ matrix in the unit square, and its estimate from 
Proposition \ref{prop:step1} 
(all curves are scaled by the values of the corresponding diagonal).
Right: upper bounds for all entries of the inverse (cf. with Figure \ref{fig:laplinv})
\label{fig:estimate_Lapl}}
\end{figure}

In the following examples, similar plots are shown, where however
all curves are scaled by the value of the
corresponding diagonal, so that the maximum value of the
column in one.

\vskip 0.1in
\begin{example}\label{ex:1}
{\rm
We consider the two-dimensional Laplacian with Dirichlet boundary conditions,
discretized by centered finite differences with a 5-point stencil, so that
$M = {\rm tridiag}(-1,\underline{2},-1)\in\RR^{10\times 10}$ and $S$ is of size $100$.
In Figure \ref{fig:estimate_Lapl} we report the values of $(S^{-1})_{:,35}$ (solid blue line),
and those of the corresponding upper bound in Proposition \ref{prop:step1}.
The agreement with the actual behavior of the column entries is
very good. Similar plots can be observed for the other columns of $S^{-1}$.
}
\end{example}

\vskip 0.1in
\begin{example}\label{ex:2}
{\rm
The second example stems from the discretization of the same operator as
in Example \ref{ex:1}, but in terms of the tensorized
Babuska-Shen basis, which uses Legendre polynomials.
The corresponding symmetric
matrix (for even degrees) is given by (see, e.g., \cite{Canuto.Simoncini.Verani.13tr}
and references therein)
$M = {\rm tridiag} ( \delta_k, \underline{\gamma_k}, \delta_k)$, where
\begin{eqnarray*}
\gamma_k&=&\frac{2}{(4k-3)(4k+1)} , \quad k=1, \ldots, n, \quad {\rm and}  \\
\delta_k&=&\frac{-1}{(4k+1)\sqrt{(4k-1)(4k+3)}}, \quad k=1, \ldots,n-1 .
\end{eqnarray*}
%
The plot of
Figure \ref{fig:legendre} reports the actual values of $(S^{-1})_{:,35}$ and
their estimates according to Proposition \ref{prop:step1}. 
Once again, the bounds appear to be fully descriptive of the 
entry pattern. 
}
\end{example}

\begin{figure}[hbt]
\centering
\includegraphics[width=2.5in,height=2.5in]{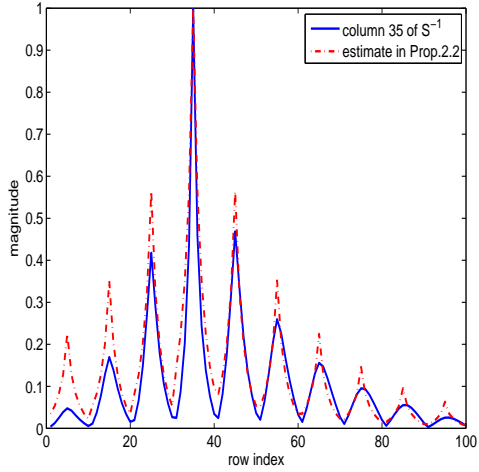}
\caption{Example \ref{ex:2}.
 Components of column $t=35$ of the inverse of the Legendre stiffness matrix
of size
$100\times 100$, and its estimate from Proposition \ref{prop:step1}.
(all curves are scaled by the values of the corresponding diagonal).
\label{fig:legendre}}
\end{figure}

 The results of Proposition \ref{prop:step1} can be manipulated to 
provide, for each $t$, more explicit estimates on the 
entries of $(S^{-1})_{:,t}$; more precisely, they are expressed in terms of the
`` index distance'' $|\ell - i| + |m-j|$.

\begin{proposition}\label{prop:step2}
For $k, t \in \{1, \ldots, n^2\}$ let
$j=\lfloor (t-1)/n\rfloor+1, i=t-n \lfloor (t-1)/n\rfloor$ and
$\ell, m$ as in (\ref{eqn:indices}).
\begin{enumerate}
{
\item[i)] Let ${\mathfrak n}_2:=|\ell - i| + |m-j|-2>0$. If $\ell\ne i$, $m\ne j$ then  
\begin{equation}\label{aux:1}
|(S^{-1})_{k,t}|
\le
\frac{1}{2\sqrt{2}}
\frac{(\lambda_{\max} - \lambda_{\min})^{{\mathfrak n}_2+2}}{(\lambda_{\max}^2 +\lambda_{\min}^2)^{{\mathfrak n}_2/2}} 
\frac {\sqrt{\lambda_{\max}^2+\lambda_{\min}^2}} {(\lambda_{\max}\lambda_{\min})^2}
\frac 1 {\sqrt{{\mathfrak n}_2}} \sqrt{\frac{2{\mathfrak n}_2}{{\mathfrak n}_2+4}}.
\end{equation}

\item[ii)] Let ${\mathfrak n}_1:=|\ell - i| + |m-j|-1>0$. If either $\ell = i$ or $m=j$ and 
${\mathfrak n}_1>0$ then
\begin{equation}\label{aux:2}
|(S^{-1})_{k,t}|
\le
\frac{1}{2\sqrt{2}}
\frac{(\lambda_{\max} - \lambda_{\min})^{{\mathfrak n}_1+1}}{(\lambda_{\max}^2 +\lambda_{\min}^2)^{{\mathfrak n}_1/2}} 
\frac {\sqrt{\lambda_{\max}^2+\lambda_{\min}^2}} {\lambda_{\max}\lambda_{\min}^2}
\frac 1 {\sqrt{{\mathfrak n}_1}} \sqrt{\frac{2{\mathfrak n}_1}{{\mathfrak n}_1+2}}.
\end{equation}
}

%
\end{enumerate}
\end{proposition}

\begin{proof} The proof is postponed to the Appendix. \end{proof}

We conclude this paragraph with a final qualitative bound {for ${\mathfrak n}_1, {\mathfrak n}_2$ large}, that emphasizes the asymptotic behavior.

\begin{proposition}\label{prop:final}
{
Let $\kappa=\lambda_{\max}/\lambda_{\min} = {\rm cond}(M)$.
\begin{enumerate}
\item[i)]
Assume $\ell, i, m, j$ are such that $\ell\ne i$, $m\ne j$ and
$ {\mathfrak n}_2=|\ell - i| + |m-j|-2>0$.
With the previous notation, it holds
$$
|(S^{-1})_{k,t}|
\le
\frac{\sqrt{\kappa^2+1}}{2\lambda_{\min}} 
\frac 1 {\sqrt{\mathfrak n}_2}.
$$
\item[ii)] Assume $\ell, i, m, j$ are such that $\ell=i$ or $m=j$ and
${\mathfrak n}_1=|\ell - i| + |m-j|-1>0$. With the previous notation, 
it holds
$$
|(S^{-1})_{k,t}|
\le
\frac{\kappa \sqrt{\kappa^2+1}}{2} 
\frac 1 {\sqrt{\mathfrak n}_1}.
$$
\end{enumerate}
}
\end{proposition}

\begin{proof}

i) The constant involving the extreme eigenvalues of $M$ satisfies
\begin{equation}\label{aux:3}
\frac{(\lambda_{\max} - \lambda_{\min})^{{\mathfrak n}+2}}{(\lambda_{\max}^2 +\lambda_{\min}^2)^{{\mathfrak n}/2}} 
\frac {\sqrt{\lambda_{\max}^2+\lambda_{\min}^2}} {(\lambda_{\max}\lambda_{\min})^2} 
\le \frac{\sqrt{\kappa^2 +1}}{\lambda_{\min}} ,
\end{equation}
where $\kappa=\lambda_{\max}/\lambda_{\min}\ge 1$.
Indeed, 
\begin{eqnarray*}
\frac{(\lambda_{\max} - \lambda_{\min})^{{\mathfrak n}+2}}{(\lambda_{\max}^2 +\lambda_{\min}^2)^{{\mathfrak n}/2}} 
\frac {\sqrt{\lambda_{\max}^2+\lambda_{\min}^2}} {(\lambda_{\max}\lambda_{\min})^2} =
\frac{\lambda_{\max}^{{\mathfrak n}+2}}{\lambda_{\max}^{\mathfrak n}}
\frac{ (1-1/\kappa)^{{\mathfrak n}+2}}{(1+1/\kappa^2)^{{\mathfrak n}/2}} 
\frac 1 {\lambda_{\min}^2\lambda_{\max}} \sqrt{ 1+\frac{1}{\kappa^2}},
\end{eqnarray*}
with $\frac{ (1-1/\kappa)^{{\mathfrak n}+2}}{(1+1/\kappa^2)^{{\mathfrak n}/2}} \le 1$.
{
Inserting \eqref{aux:3}  into \eqref{aux:1} and
noticing that $\frac{{\mathfrak n}_2}{{\mathfrak n}_2/2+1} \le 2$ yield the result. 

The proof of ii) goes along the same lines of i) after observing that it holds 
$$
 \frac{(\lambda_{\max} - \lambda_{\min})^{{\mathfrak n}+2}}{(\lambda_{\max}^2 +\lambda_{\min}^2)^{{\mathfrak n}/2}} 
\frac {\sqrt{\lambda_{\max}^2+\lambda_{\min}^2}} {\lambda_{\max}\lambda_{\min}^2} 
\le \frac{\lambda_{\max}\sqrt{\kappa^2 +1}}{\lambda_{\min}}=\kappa\sqrt{\kappa^2 +1}. 
$$
}
\end{proof}

We remark that a result in a similar direction was reported in \cite[Theorem 4.5]{MeurantJul.1992} for
the Laplacian matrix,
although in there, an explicit dependence on the problem dimension arises, together with a 
more involved dependence on the discretization grid.

In terms of the indices of $S^{-1}$, 
our bound shows that
\begin{eqnarray*}
|(S^{-1})_{k,t}| &=& |(S^{-1})_{\ell+n(m-1),i+n(j-1)} |  \\
&\le &
\gamma_0 
\frac 1 {\sqrt{|\ell -i|+|m-j|-2}}  \\
&=&
\gamma_0 
\frac 1 {\sqrt{|k-t -n( \lfloor \frac {(k-1)} n\rfloor - \lfloor \frac {(t-1)} n\rfloor )|+
|\lfloor \frac {(k-1)} n\rfloor - \lfloor \frac {(t-1)} n\rfloor |-2}} .
\end{eqnarray*}

For instance,
for all the elements on the secondary diagonal, satisfying $k+t=n^2$, we
obtain
$$
|(S^{-1})_{k,t}| \le 
\gamma_0 
\frac 1 {( |n^2-2t - n (\lfloor \frac{n^2-t-1}{n}\rfloor - 
\lfloor \frac{t-1}{n}\rfloor)|+ | \lfloor \frac{n^2-t-1}{n}\rfloor - 
\lfloor \frac{t-1}{n}\rfloor| -2)^{\frac 1 2 }}.
$$

In the qualitative bound of Proposition \ref{prop:final}
the asymptotic term does not depend on the actual
entries of $S^{-1}$, but only on the position in the underlying grid. 
This property reflects similar considerations obtained
for point-wise estimates in the context of the discrete Laplacian.
Indeed, for the discrete Green function $G_h$ on the discrete $N$-dimensional grid $R_h$,
it was shown in \cite{Bramble.Thomee.69} that there exist constants $h_0$ and $C$ such
that for $h\le h_0$, $x,y\in R_h$, 
\begin{eqnarray}\label{eqn:bramble}
G_h(x,y) \le 
\left \{ \begin{array}{ll} 
C \log \frac{C}{|x-y|+h} & {\rm if } N=2 \\
 \frac{C}{(|x-y|+h)^{N-2}} & {\rm if } N\ge 3 .
\end{array} \right .
\end{eqnarray}
Our computable bound in Proposition~\ref{prop:final} shows that the entries depend on the
inverse square root of the distance, whereas in (\ref{eqn:bramble}) an asymptotic
(slower) logarithmic dependence on the distance is reported for the two-dimensional case.

We also notice that other bounds are available that use different distance concepts; for instance,
in \cite[Theorem 3.4]{Benzi2007} the decay pattern of certain matrix functions is
described by means of graph theory, in terms of distance\footnote{Defined as the
length of the shortest directed path connecting the two nodes.} between nodes of a digraph, 
where the nodes are the entry indices in the matrix inverse.

%

\section{On the decay of the Cholesky factor}\label{sec:Cholfactor}
When preconditioning a large algebraic linear system, a-priori information
on the decay properties of the inverse of the Cholesky factor of $S$ may
be important to guide the computation of incomplete factorizations. Indeed,
assuming that $S=LL^{\top}$ is the Cholesky factorization of $S$,
if the entries of  $S^{-1}$ decay away from the main diagonal with a certain
pattern, we expect that also the factor
inverse $L^{-\top}$ will show a similar pattern. This fact was proved in
\cite{Benzi.Tuma.00} for banded matrices $S$ by using the decay rate
in (\ref{eqn:demko}).
With the same technical tools, we
generalize this result to our decay pattern, under the assumption that
 $S$ has a bandwidth $b$.

\begin{proposition}
Assume $S$ is $b$-banded, with diagonal elements not greater than one, and 
let $S=LL^\top$. With the previous notation, for ${\mathfrak n}={\mathfrak n}_i >0$, 
$i=1,2$,
$$
|(L^{-\top})_{k,t}| \le
\gamma_0  \frac b {\sqrt{\mathfrak n(k,t)}}  .
$$
\end{proposition}
\begin{proof}
We have
$$
|(L^{-\top})_{k,t}| \le \sum_{r=t}^{t+b-1} | (S^{-1})_{k,r}| \, |L_{r,t}|
\le \gamma_0 \sum_{r=t}^{t+b-1} \frac 1 {\sqrt{\mathfrak n(k,r)}} \leq \gamma_0  
\frac b {\sqrt{\mathfrak n(k,t)}}  ,
$$
where we used the inequality 
${\mathfrak n(k,t)}\leq {\mathfrak n(k,r)}$ for $k\leq t\leq r \leq t+b-1$.
\end{proof}

We notice that a slightly sharper upper bound could be obtained by first
using the bound of Proposition~\ref{prop:step1}, however the final asymptotic
dependence with respect to $\mathfrak n$ would still be the same.

The estimate for the entries of $L^{-1}$ could be used in the design
of linear system preconditioners by means of approximate inverses 
\cite{Benzi.Tuma.00},\cite{Benzi.prec.02}.
Indeed, not only a decay pattern occurs away from the
diagonal, but many tiny values appear within the bandwidth. Therefore,
a threshold-based dropping strategy could be used in conjunction with a band-based
procedure, to a-priori increase the sparsity of the explicit
approximate inverse. Similar considerations can guide the 
design of quasi-orthogonal polynomial bases, as those developed in,
e.g., \cite{Canuto.Simoncini.Verani.13tr},\cite{Challacombe.99}.

\section{More general settings}\label{sec:gen}
The results of the previous sections can be generalized in a number of
ways. For instance, we can allow the symmetric and positive matrix $M$ in (\ref{eqn:main})
to be generally $b$-banded, so that Proposition~\ref{prop:Freund}
can be used in its full generality. The resulting estimate is reported
below. Its proof is omitted as it is analogous to that of Proposition \ref{prop:step1}.

\begin{proposition}\label{prop:step1band}
Let $M$ be a real symmetric and positive definite matrix of size $n$ and bandwidth $b$.
 For $k, t\in \{1, \ldots, n^2\}$,
let
$$
j=\lfloor (t-1)/n\rfloor+1,\quad i=t-n \lfloor (t-1)/n\rfloor
$$
 and $\ell, m$ as in (\ref{eqn:indices}).
With the notation above, the following holds.

i) If $i\ne \ell$ and $j\ne m$, then
$$
|(S^{-1})_{k,t}|  
\le
\frac{1}{2\pi}
\frac{64}{|\lambda_{\max}-\lambda_{\min}|^2} \int_{-\infty}^{\infty} 
 \left ( \frac{R^2}{(R^2-1)^2}\right )^2
\left ( \frac{1}{R}\right )^{|i-\ell|/b+|j-m|/b-2} {\rm d}\omega ;
$$
ii) If either $i= \ell$  or $j= m$, then
$$
|(S^{-1})_{k,t}|  
\le
\frac{1}{2\pi}
\frac{8}{|\lambda_{\max}-\lambda_{\min}|} \int_{-\infty}^{\infty} 
\frac 1 {\sqrt{\lambda_{\min}^2 + \omega^2}}
 \frac{R^2}{(R^2-1)^2}
\left ( \frac{1}{R}\right )^{|i-\ell|/b+|j-m|/b-1} {\rm d}\omega ;
$$
iii) If both $i= \ell$  and $j= m$, then
$$
|(S^{-1})_{k,t}|  
\le
\frac{1}{2\pi}
\int_{-\infty}^{\infty} \frac 1 {\lambda_{\min}^2 + \omega^2} {\rm d}\omega = \frac 1 {2\lambda_{\min}}.
$$
\end{proposition}

\begin{figure}[hbt]
\centering
\includegraphics[width=2.5in,height=2.5in]{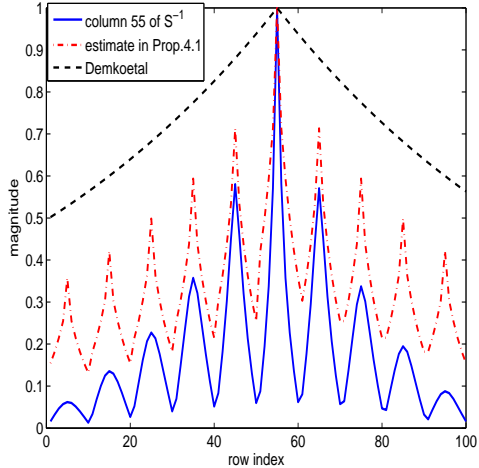}
\caption{Example \ref{ex:penta}.
 Components of column $t=55$ of the inverse of the Laplace stiffness matrix
of size
$100\times 100$ with a 9-point stencil discretization, 
and its estimate from Proposition \ref{prop:step1band}.
(all curves are scaled by the values of the corresponding diagonal.)
\label{fig:penta}}
\end{figure}

\vskip 0.1in
\begin{example}\label{ex:penta}
{\rm
We consider the $100\times 100$ matrix $S$ stemming from
the discretization of the two-dimensional Laplace operator
by means of a more accurate discretization (9-point stencil) of the
one-dimensional derivative than in Example \ref{ex:1}. This gives
$$
M={\rm pentadiag}\left (\frac 1{12}, -\frac 4 3, \underline{\frac{15}{6}}, -\frac 4 3, \frac 1{12}\right ),
$$ 
which has bandwidth $b=2$. The plot in Figure \ref{fig:penta} (for column $t=55$ of
the inverse $S^{-1}$) shows that
the estimate of Proposition~\ref{prop:step1band} is able to capture the 
oscillating behavior, but somewhat fails to follow the asymptotic pattern of the inverse, predicting
a slower decay. (We recall here that all values were scaled to be one at the $t$-th
component.)  As already mentioned, this seems
to be due to the weakness of the exponential bound in Proposition \ref{prop:Freund}
for a larger bandwidth. We explicitly observe that also the monotonic classical
bound (\ref{eqn:demko}) does not seem to closely match the actual asymptotic decay pattern; we
should keep in mind that in this case, $S$ has bandwidth $b\cdot n = 20$, which seems to
also significantly deteriorate the classical estimate.
}
\end{example}

Another generalization is obtained by assuming that
$S=S_g$ can be written as in (\ref{eqn:M1M2}), with $M_1$, $M_2$ 
symmetric and positive definite square matrices, of size $n_1$ and $n_2$, respectively, so that
$S_g$ is of size $n_1n_2\times n_1n_2$. Following the derivation in 
Section~\ref{sec:decay}, the elements of
each column $t$ of the inverse can be derived as the elements
of the solution matrix ${\cal X}$ to the {\it Sylvester} equation
\begin{eqnarray}\label{eqn:sylv}
M_1 {\cal X} + {\cal X} M_2 = {\cal E} .
\end{eqnarray}
The following result generalizes one of the cases of Proposition \ref{prop:step1}. The other case
can be derived analogously.

\begin{proposition}
Assume $M_1, M_2$ are symmetric positive definite and tridiagonal matrices.
Let $\delta_{1,2}=
(\lambda_{\max}(M_1)-\lambda_{\min}(M_1))
(\lambda_{\max}(M_2)-\lambda_{\min}(M_2))$.
For $k, t\in \{1, \ldots, n_1n_2\}$ let
$j=\lfloor (t-1)/n_1\rfloor+1,\quad i=t-n_1 \lfloor (t-1)/n_1\rfloor$,
$\ell=\lfloor (k-1)/n_1\rfloor+1,\quad m=t-n_1 \lfloor (k-1)/n_1\rfloor$,
be such that
$\ell \ne i$ and $m\ne j$. Then it holds that
$$
|(S_g^{-1})_{k,t}|  
\le
\frac{1}{2\pi}
\frac{64}{\delta_{1,2}} \int_{-\infty}^{\infty} 
  \frac{R_1^2}{(R_1^2-1)^2} 
  \frac{R_2^2}{(R_2^2-1)^2}
\left ( \frac{1}{R_1}\right )^{|i-\ell|-1}
\left ( \frac{1}{R_2}\right )^{|j-m|-1} {\rm d}\omega ,
$$
with $R_1$ and $R_2$ are defined as in Proposition \ref{prop:Freund} for
each of the spectra of $M_1$ and $M_2$, respectively.
\end{proposition}

\begin{proof}
We can write the solution to the Sylvester equation in closed form
as (see, e.g., \cite{Horn.Johnson.91})
\begin{eqnarray*}
{\cal X}_t & = &
\frac{1}{2\pi}
\int_{-\infty}^{\infty} (\imath \omega I + M_1)^{-1} e_i e_j^\top
(\imath \omega I + M_2)^{-*} {\rm d}\omega .
\end{eqnarray*}
To evaluate $|({\cal X}_t)_{\ell,m}|$
we can then apply again Proposition~\ref{prop:Freund} to each of the
inner term, namely to
$|e_\ell^\top (\imath \omega I + M_1)^{-*} e_i|$,
$|e_j^\top (\imath \omega I + M_2)^{-*} e_m|$. The final result thus 
follows as in Proposition \ref{prop:step1}.
\end{proof}

Finally, we observe that
the two generalizations above could be combined, giving estimates for
the entries of the inverse when $M_1$ and $M_2$ have different bandwidth.

\section{Conclusions}
We have characterized the decay pattern of the inverse of
banded matrices that can be written
as the sum of two Kronecker products, in which each of the matrices is symmetric
positive definite and banded. Our results explain the non-monotonic
(oscillating) pattern commonly observed in these inverses, while providing
upper bounds that can be sharp, especially for low bandwidth.
We also showed that corresponding results can be obtained for
more general Kronecker-type matrices with different banded matrices
$M_1$ and $M_2$.

\section*{Acknowledgements}
The authors would like to thank Leonid Knizhnerman for carefully
reading an earlier draft of this manuscript, and Michele Benzi for
helpful discussions.
{The first and third authors were partially supported by 
the Italian research fund INdAM-GNCS 2013
``Aspetti emergenti nello studio di strategie adattative
per problemi differenziali''.}
The second author was partially supported by the University of Bologna through the
FARB Project ``Metodi matematici per l'esplorazione ambientale sostenibile''.

\section*{Appendix}
In this appendix we prove Proposition \ref{prop:step2}.

\begin{proof}
i) We set ${\cal X}_{\ell m} := e_\ell^\top {\cal X}_t e_m = (S^{-1})_{k,t}$.
From the result of Proposition \ref{prop:step1}, we need to bound the
integrand in a way that the integral still converges.

We observe that
\begin{eqnarray*}
\frac 1 R \le \frac 1 \alpha &= &
\frac{\lambda_{\max} - \lambda_{\min}}{|\lambda_{1}| +|\lambda_{2}|} \\
&\le&
\frac{\lambda_{\max} - \lambda_{\min}}{(\lambda_{\max}^2 +\lambda_{\min}^2+2\omega^2)^{1/2}} =
\frac{\lambda_{\max} - \lambda_{\min}}{(\lambda_{\max}^2 +\lambda_{\min}^2)^{1/2}} 
\frac{1}{\left(1+ \frac{2\omega^2}{\lambda_{\max}^2 +\lambda_{\min}^2}\right)^{1/2}}  .
\end{eqnarray*}
Moreover, after some algebraic calculation, it follows that
$R-1/R = 2\sqrt{\alpha^2-1}$, and 
since $\alpha^2-1\ge (2\omega^2+2\lambda_{\max}\lambda_{\min})/(\lambda_{\max}-\lambda_{\min})^2$,
\begin{eqnarray}\label{eqn:R}
\frac{R^2}{(R^2-1)^2} = \frac 1 {4(\alpha^2-1)} \le 
\frac 1 8 
\frac{|\lambda_{\max} - \lambda_{\min}|^2}{\omega^2+\lambda_{\max}\lambda_{\min}}  .
\end{eqnarray}
Therefore, letting ${{\mathfrak n}_2}={|i-\ell|+|j-m|-2}$,
\begin{eqnarray*}
|{\cal X}_{\ell m}| & \le &
\frac{1}{2\pi}
\frac{64}{|{
\lambda_{\max}-\lambda_{\min}}|^2} \int_{-\infty}^{\infty} 
 \left ( \frac{R^2}{(R^2-1)^2}\right )^2
\left ( \frac{1}{R}\right )^{{\mathfrak n}_2} 
{\rm d}\omega \\
& \le & 
\frac{1}{2\pi}
\frac{64 |\lambda_{\max} - \lambda_{\min}|^4 }{{64}|\lambda_{\max}-\lambda_{\min}|^2} 
\left(\frac{\lambda_{\max} - \lambda_{\min}}{(\lambda_{\max}^2 +\lambda_{\min}^2)^{\frac 1 2}} 
\right)^{{\mathfrak n}_2} 
\int_{-\infty}^{\infty} 
\frac{1}{(\omega^2+\lambda_{\max}\lambda_{\min})^2}  
\left (\frac{1}{1+ \frac{2\omega^2}{\lambda_{\max}^2 +\lambda_{\min}^2}} 
\right)^{\frac{{\mathfrak n}_2}{2}}
{\rm d}\omega \\
& \le & 
\frac{1}{2\pi} 
\frac{(\lambda_{\max} - \lambda_{\min})^{{{\mathfrak n}_2}+2}}{(\lambda_{\max}^2 +\lambda_{\min}^2)^{\frac{{\mathfrak n}_2}{2}}} 
\int_{-\infty}^{\infty} 
\frac{1}{(\omega^2+\lambda_{\max}\lambda_{\min})^2}  
\left (\frac{1}{1+ \frac{2\omega^2}{\lambda_{\max}^2 +\lambda_{\min}^2}} 
\right)^{\frac{{\mathfrak n}_2}{2}}
{\rm d}\omega .
\end{eqnarray*}
Since
$$
\frac{1}{\omega^2+\lambda_{\max}\lambda_{\min}}   =
\frac{1}{\lambda_{\max}\lambda_{\min}} 
\frac 1 {\frac{2\omega^2}{2\lambda_{\min}\lambda{\max}} +1}
\le 
\frac{1}{\lambda_{\max}\lambda_{\min}} 
\frac 1 {\frac{2\omega^2}{\lambda_{\min}^2+\lambda_{\max}^2} +1} ,
$$
we bound the entry further as
$$
|{\cal X}_{\ell m}|  \le 
\frac{1}{2\pi}
\frac{(\lambda_{\max} - \lambda_{\min})^{{{\mathfrak n}_2}+2}}{(\lambda_{\max}^2 +\lambda_{\min}^2)^{{{\mathfrak n}_2}/2}} 
\frac 1 {(\lambda_{\max}\lambda_{\min})^2}
\int_{-\infty}^{\infty} 
\left (\frac{1}{1+ \frac{2\omega^2}{\lambda_{\max}^2 +\lambda_{\min}^2}} \right)^{{{\mathfrak n}_2}/2+2} 
{\rm d}\omega .
$$
We next estimate the integral. 
Let $\tau=\sqrt{ \frac{2}{\lambda_{\max}^2+\lambda_{\min}^2}} \omega$, so that
$d\tau = \sqrt{ \frac{2}{\lambda_{\max}^2+\lambda_{\min}^2}} d\omega$.
Then
$$
\int_{-\infty}^{\infty} 
\left (\frac{1}{1+ \frac{2\omega^2}{\lambda_{\max}^2 +\lambda_{\min}^2}} \right)^{{{\mathfrak n}_2}/2+2} d\omega
=
\sqrt{2} \sqrt{\lambda_{\max}^2+\lambda_{\min}^2}  
\int_{0}^{\infty} \left(\frac{1}{1+\tau^2} \right)^{{{\mathfrak n}_2}/2+2} d\tau .
$$
The integral above  is half the Beta function 
${\cal B}(\frac 1 2, \frac{{{\mathfrak n}_2}+3}{2})$ (\cite[formula 8.38.2]{Gradshteyn1980}).
It is known that for ${{\mathfrak n}_2}$ large,  
${\cal B}(\frac 1 2, \frac{{{\mathfrak n}_2}+3}{2}) \approx \Gamma(1/2) ( ({{\mathfrak n}_2}+3)/2)^{-1/2}$.
However, we can provide an explicit bound for the integral. We recall that
$(1+\tau)^k \ge 1+ k \tau$, for all $\tau>-1$. Then, using the
change of variable $\tau=s/\sqrt{{{\mathfrak n}_2}}$, we can write
\begin{eqnarray*}
\int_{0}^{\infty} \frac{1}{(1+\tau^2)^{{{\mathfrak n}_2}/2+2}} d\tau & \le &
\int_{0}^{\infty} \frac{1}{1+({{\mathfrak n}_2}/2+2)\tau^2} d\tau \\
&=&
\frac 1 {\sqrt{{{\mathfrak n}_2}}}\int_{0}^{\infty} \frac{1}{1+({{\mathfrak n}_2}/2+2)s^2/{{\mathfrak n}_2}} ds\\
&=&
\frac 1 {\sqrt{{{\mathfrak n}_2}}} \sqrt{\frac{{{\mathfrak n}_2}}{{{\mathfrak n}_2}/2+{ 2}}} 
{\rm atan}\left ( s \sqrt{\frac{{{\mathfrak n}_2}/2+2}{{{\mathfrak n}_2}}} \right ) {\mid}_0^{\infty}  =
\frac{\pi}{2} \frac 1 {\sqrt{{{\mathfrak n}_2}}} 
\sqrt{\frac{{{\mathfrak n}_2}}{{{\mathfrak n}_2}/2+{ 2}}} 
\end{eqnarray*}
{ thus yielding 
$$
|{\cal X}_{\ell m}|
\le
\frac{1}{2\sqrt{2}}
\frac{(\lambda_{\max} - \lambda_{\min})^{{\mathfrak n}_2+2}}{(\lambda_{\max}^2 +\lambda_{\min}^2)^{{\mathfrak n}_2/2}} 
\frac {\sqrt{\lambda_{\max}^2+\lambda_{\min}^2}} {(\lambda_{\max}\lambda_{\min})^2}
\frac 1 {\sqrt{{\mathfrak n}_2}} \sqrt{\frac{{\mathfrak n}_2}{{\mathfrak n}_2/2+2}}.
$$
}

ii) If $i=\ell$ or $j=m$, then Proposition \ref{prop:step1}(ii) applies.
We set  $\mathfrak{n}_1=|i-\ell|+|j-m|-1$.
Using { again} (\ref{eqn:R}) and the bound on $1/R$, the following bound holds
\begin{eqnarray*}
&&\int_{-\infty}^{\infty} 
\frac 1 {\sqrt{\lambda_{\min}^2 + \omega^2}}
 \frac{R^2}{(R^2-1)^2}
\left ( \frac{1}{R}\right )^{\mathfrak{n}_1} {\rm d}\omega 
\\
&& \le
\frac 1 8 
\frac{(\lambda_{\max}-\lambda_{\min})^{\mathfrak{n}_1}}{(\lambda_{\max}^2+\lambda_{\min}^2)^{\frac{\mathfrak{n}_1}{2}}}
|\lambda_{\max}-\lambda_{\min}|
\int_{-\infty}^{\infty} 
\frac 1 {\sqrt{\lambda_{\min}^2 + \omega^2}}
\frac 1 {\omega^2 + \lambda_{\max}\lambda_{\min}}
\frac 1 {\left( 1+ \frac{2\omega^2}{\lambda_{\min}^2 + 
\lambda_{\min}^2}\right)^{\frac{\mathfrak{n}_1}{2}}}
{\rm d}\omega
\\
&& \le
\frac 1 8 
\frac{(\lambda_{\max}-\lambda_{\min})^{\mathfrak{n}_1+1}}{(\lambda_{\max}^2+\lambda_{\min}^2)^{\frac{\mathfrak{n}_1}{2}}}
\frac{1}{\lambda_{\max}\lambda_{\min}} 
\int_{-\infty}^{\infty} 
\frac 1 {\sqrt{\lambda_{\min}^2 + \omega^2}}
\frac 1 {\frac{2\omega^2}{\lambda_{\min}^2+\lambda_{\max}^2} +1} 
\frac 1 {\left( 1+ \frac{2\omega^2}{\lambda_{\min}^2 + \lambda_{\min}^2}\right)^{\frac{\mathfrak{n}_1}{2}}}
{\rm d}\omega
\\
&& =
\frac 1 8 
\frac{(\lambda_{\max}-\lambda_{\min})^{\mathfrak{n}_1+1}}{(\lambda_{\max}^2+\lambda_{\min}^2)^{\frac{\mathfrak{n}_1}{2}}}
\frac{1}{\lambda_{\max}\lambda_{\min}} 
\int_{-\infty}^{\infty} 
\frac 1 {\sqrt{\lambda_{\min}^2 + \omega^2}}
\frac 1 {\left( 1+ \frac{2\omega^2}{\lambda_{\min}^2 + \lambda_{\min}^2}\right)^{\frac{\mathfrak{n}_1}{2}+1}}
{\rm d}\omega\\
&& { \le 
\frac 1 8 
\frac{(\lambda_{\max}-\lambda_{\min})^{\mathfrak{n}_1+1}}{(\lambda_{\max}^2+\lambda_{\min}^2)^{\frac{\mathfrak{n}_1}{2}}}
\frac{1}{\lambda_{\max}\lambda^2_{\min}} 
\int_{-\infty}^{\infty} 
\frac 1 {\left( 1+ \frac{2\omega^2}{\lambda_{\min}^2 + \lambda_{\min}^2}\right)^{\frac{\mathfrak{n}_1}{2}+1}}
{\rm d}\omega}
\end{eqnarray*}
{ where in the last inequality we used $\sqrt{\lambda_{\min}^2 + \omega^2}\geq \lambda_{\min}$. 

Finally, estimating the integral in the above inequality in the same way as in the proof of i), we get
$$
|{\cal X}_{\ell m}|
\le
\frac{1}{2\sqrt{2}}
\frac{(\lambda_{\max} - \lambda_{\min})^{{\mathfrak n}_1+1}}{(\lambda_{\max}^2 +\lambda_{\min}^2)^{{\mathfrak n}_1/2}} 
\frac {\sqrt{\lambda_{\max}^2+\lambda_{\min}^2}} {\lambda_{\max}\lambda_{\min}^2}
\frac 1 {\sqrt{{\mathfrak n}_1}} \sqrt{\frac{{\mathfrak n}_1}{{\mathfrak n}_1/2+1}}. 
$$
} 
\end{proof}

\end{document}